\documentclass[11pt, reqno]{amsart}
\usepackage[all]{xy}
\usepackage{a4wide,fullpage}  \usepackage{amssymb}
\usepackage{amsthm}
\usepackage{amsmath,cancel}
\usepackage{amscd,enumitem}
\usepackage{verbatim}
\usepackage{float}
\usepackage{color,mdframed}
\usepackage{bbm}
\usepackage[mathscr]{eucal}
\usepackage[all]{xy}
\usepackage{bm}
\usepackage[hang]{footmisc}
\usepackage[draft]{hyperref}
\usepackage{mathtools}
\usepackage{calligra}
\usepackage{stmaryrd}
\DeclareMathAlphabet{\mathcalligra}{T1}{calligra}{m}{n}
\newtheorem{theorem}{Theorem}

\newtheorem{lemma}[theorem]{Lemma}

\newtheorem{proposition}[theorem]{Proposition}
\newtheorem*{theorem*}{Theorem}

\theoremstyle{definition}
\newtheorem{definition}[theorem]{Definition}

\newtheorem*{remark}{Remark}
\newtheorem*{remarks}{Remarks}

\newcommand{\R}{\mathbb{R}}
\newcommand{\Q}{\mathbb{Q}}
\newcommand{\Z}{\mathbb{Z}}
\newcommand{\N}{\mathbb{N}}
\newcommand{\C}{\mathbb{C}}

\renewcommand{\H}{\mathbb{H}}

\newcommand{\F}{\mathbb{F}}

\newcommand{\tr}{\operatorname{tr}}

\newcommand{\sgn}{\operatorname{sgn}}

\newcommand{\Aut}{\operatorname{Aut}}

\newcommand{\re}{\operatorname{Re}}
\newcommand{\im}{\operatorname{Im}}

\newcommand{\SL}{{\text {\rm SL}}}

\newcommand{\ST}{\operatorname{ST}}

\renewcommand{\pmod}[1]{\  \,  \left(  \operatorname{mod} \,  #1 \right)}
\newcommand{\Pmod}[1]{\  \,  (  \operatorname{mod} \,  #1 )}

\DeclareMathAlphabet{\mathpzc}{OT1}{pzc}{m}{it}

\numberwithin{equation}{section}
\numberwithin{theorem}{section}
\allowdisplaybreaks

\begin{document}
\title[Sato--Tate in arithmetic progressions]{Odd moments for the trace of Frobenius and the Sato--Tate conjecture in arithmetic progressions}
\author{Kathrin Bringmann}
\address{University of Cologne, Department of Mathematics and Computer Science, Weyertal 86-90, 50931 Cologne, Germany}
\email{kbringma@math.uni-koeln.de}
\author{Ben Kane}
\address{Department of Mathematics, University of Hong Kong, Pokfulam, Hong Kong}
\email{bkane@hku.hk}
\author{Sudhir Pujahari}
\address{Institute of Mathematics, Warsaw University, Banacha 2, 02-97 Warsaw, Poland. 
 \ \ \ \ \ \ \ \ \ \ \ \ \ \ \  \ \ \ \ \ \ \ \ \ \ \ \ \ \ \ \ \ \ \ \ \ \ \ \ \ \ \ \ \ \ \ \ \ \ \ \ \ \ \ \ \ \ \ \ \ \ \ \ \ \ \
\\
School of Mathematical Sciences, National Institute of Science Education and Research, Bhubaneswar, HBNI, P. O. Jatni, Khurda 752050, Odisha, India }
\email{spujahari@mimuw.edu.pl/spujahari@gmail.com}

\date{\today}
\begin{abstract}
In this paper, we consider the moments of the trace of Frobenius of elliptic curves if the trace is restricted to a fixed arithmetic progression. We determine the asymptotic behavior for the ratio of the $(2k+1)$-th moment to the zeroeth moment as the size of the finite field $\F_{p^r}$ goes to infinity. These results follow from similar asymptotic formulas relating sums and moments of Hurwitz class numbers where the sums are restricted to certain arithmetic progressions. As an application, we prove that the distribution of the trace of Frobenius in arithmetic progressions is equidistributed with respect to the Sato--Tate measure.
\end{abstract}
\maketitle

\section{Introduction and statement of results}\label{sec:intro}

For an elliptic curve $E$ over the finite field $\F_{p^r}$ with $p^r$ elements, there is a natural endomorphism called the Frobenius map. Letting $\tr(E):=p^r+1-\#E(\F_{p^r})$ be the {\it trace of Frobenius}, Hasse \cite{Hasse} proved that $|\tr(E)|\le 2p^{\frac r2}$. Hence
\[
x_E:=\frac{\tr(E)}{2p^{\frac r2}}\in[-1,1].
\]
It is natural to investigate the distribution of $x_{E}\in [-1,1]$ as $p^r\to\infty$. Birch \cite{Birch} related this distribution to the \begin{it}Sato--Tate measure\end{it} $\mu_{\ST}(X):=\int_{X} f_{\ST}(x)dx$, where $X\subseteq\R$ is a Lebesgue measurable set and
\[
f_{\operatorname{ST}}(x):= \begin{cases}
\frac{2}{\pi}\sqrt{1-x^2} & \text{if }x \in [-1,1], \\
0 & \text{otherwise}.
\end{cases}
\]
Specifically, letting $E_{\alpha,\beta}$ denote the elliptic curve in Weierstrass form $y^2=x^3-\alpha x -\beta$ for $\alpha,\beta\in\F_p$, Birch \cite{Birch} proved that for $-1\leq a\leq b\leq 1$ we have
\[
\lim_{p\to\infty} \operatorname{Pr}\left(a\leq x_{E_{\alpha,\beta}}\leq b: \alpha,\beta\in \F_p\right) = \mu_{\ST}([a,b]).
\]
Here, as usual, for a finite set $\mathcal{E}$ and a statement  $\mathcal{S}(x)$ (which is true or false for each $x\in\mathcal{E}$), we define the probability of the event $\mathcal{S}$ in the sample space $\mathcal{E}$  to be
\[
\operatorname{Pr}\left(\mathcal{S}(x):x\in \mathcal{E}\right):=\frac{\#\left\{x\in \mathcal{E}: \mathcal{S}(x)\text{ holds}\right\}}{\#\mathcal{E}}.
\]
We say that the set of $x_E$ for $E\in \{E_{\alpha,\beta}:\alpha,\beta\in\F_p\}$ becomes \begin{it}equidistributed\end{it} with respect to the Sato--Tate measure as $p\to\infty$. Birch's result has been extended to the case where $\F_p$ is replaced with an arbitrary finite field $\F_{p^r}$ (see the work of Brock and Granville \cite{BrockGranville} as well as \cite{KatzSarnak}). Instead of considering the sample space defined by the coefficients of the cubic polynomial, we investigate subsets of the set $\mathcal{E}(p^r)$ of isomorphism classes of elliptic curves over $\F_{p^r}$. Specifically, for $m\in\Z$ and $M\in \N:=\{1,2,3,\dots\}$, we generalize Birch's result by considering the sample space
\[
\mathcal{E}_{m,M}(p^r):=\{E\in\mathcal{E}(p^r):\tr(E)\equiv m\pmod{M}\}.
\]
\begin{theorem}\label{thm:SatoTateArithmetic}
	Let $m\in\Z$ and $M,r\in\N$. Restricting to $E\in\mathcal{E}_{m,M}$ over $\F_{p^r}$, the normalized traces of Frobenius $x_E$ are equidistributed with respect to the Sato--Tate measure in the interval $[-1,1]$ as $p \to \infty$. Specifically, we have for $-1\leq a\leq b\leq 1$
\rm
	\[
		\lim_{p\to\infty} \operatorname{Pr}\left(a\leq x_E\leq b: E\in\mathcal{E}_{m,M}(p^r)\right) = \mu_{\ST}([a,b]).
	\]
\end{theorem}

It is also natural to consider the analogous question for the probabilities
\begin{equation}\label{eqn:Prgeneral}
\operatorname{Pr}\left(a\leq x_E\leq b\text{ and }E\in \mathcal{E}_{m,M}(p^r): E\in \mathcal{E}(p^r)\right)
\end{equation}
on the larger sample space $\mathcal{E}(p^r)$. Since the distribution of $x_E$ in $[-1,1]$ is the same on the sample spaces $\mathcal{E}(p^r)$ and $\mathcal{E}_{m,M}(p^r)$ as $p\to\infty$, Theorem \ref{thm:SatoTateArithmetic} shows that as $p\to\infty$ the set of $x_E$ becomes equidistributed independently of the distribution of $\tr(E)$  modulo $M$. The distribution of $\tr(E)$ modulo $M$ was studied by Castryck and Hubrechts \cite{CastryckHubrechts}. In particular, for fixed $m$ and $M$, they determined \cite[Theorem 1]{CastryckHubrechts} the asymptotic behavior of
\begin{equation}\label{eqn:PrEs}
\operatorname{Pr}\left(E\in \mathcal{E}_{m,M}(p^r): E\in \mathcal{E}(p^r)\right),
\end{equation}
up to an error of size $O_M(p^{-\frac{r}{2}})$. Given the independence of the distribution of $x_E$ and the distribution of $\tr(E)$ modulo $M$ as $p\to\infty$, \eqref{eqn:Prgeneral} converges to \eqref{eqn:PrEs} times the Sato--Tate distribution. While the limit as $p\to\infty$ of \eqref{eqn:PrEs} (and hence \eqref{eqn:Prgeneral}) does not exist in general, the limit exists under certain restrictions. We list one such case in the following theorem.
 
\begin{theorem}\label{thm:SatoTateArithmetic2}
	Let $m\in\Z$ and $M\in\N$. Then for each $j\in\Z$ with $\gcd(j,2M)=1$ there exists $c_{j,m,M} \in \R_{>0}$ such that the normalized traces $x_E$ ($E\in\mathcal{E}_{m,M}(p)$) are each equidistributed in $[-1,1]$ within the larger space $\mathcal{E}(p)$ with respect to the measure $c_{j,m,M} \mu_{\operatorname{ST}}$ as $p\to\infty$ with $p\equiv j\Pmod{4M^2}$. Specifically, for $-1\leq a\leq b\leq 1$ we have 
	\[
		\lim_{\substack{p\to\infty\\ p\equiv j\pmod{4M^2}}} \operatorname{Pr}\left(a\leq x_E\leq b\text{ and }E\in \mathcal{E}_{m,M}(p):E\in \mathcal{E}(p)\right) = c_{j,m,M}\mu_{\ST}([a,b]).
	\]
\end{theorem}
\begin{remark}
With a little more work, one should be able to determine the limiting distribution coming from taking $p\to\infty$ in \eqref{eqn:Prgeneral} in each of the congruence classes by using the evaluation of \eqref{eqn:PrEs} in \cite[Theorem 1]{CastryckHubrechts}. 
\end{remark}

In order to investigate the distributions in Theorems \ref{thm:SatoTateArithmetic} and \ref{thm:SatoTateArithmetic2}, we use the method of moments (see Subsection \ref{sec:momentmethod}), which relies on the fact that a reasonable probability measure is determined by its moments. We investigate closely-related weighted moments that converge to the unweighted moments as $p\to\infty$ (see Subsection \ref{sec:weightedist}).  To describe the weighting, denote by $\Aut_{\F_{p^r}}(E)$ the automorphism group of $E/\F_{p^r}$ and set $\omega_{E}:=\#\Aut_{\F_{p^r}}(E)$. 
In particular, we show that for $\nu\in\N_0:=\N\cup \{0\}$ the weighted moments of $x_E$ 
\begin{equation}\label{eqn:Sratio}
\sum_{\substack{E/\F_{p^r}\\ \tr(E)\equiv m\pmod{M}}} \frac{x_E^{\nu}}{\omega_E}=\frac{1}{2^{\nu}p^{\frac{r\nu}{2}}} S_{\nu,m,M}(p^r)
\end{equation}
converge to the moments of $\mu_{\operatorname{ST}}$, where 
\begin{equation}\label{eqn:SmMkdef}
S_{\nu,m,M}(p^r):=\sum_{\substack{E/\F_{p^r}\\ \tr(E)\equiv m\pmod{M}}} \frac{\tr(E)^{\nu}}{\omega_E}.
\end{equation}
The weighted even moments $\nu\in 2\N_0$ have been studied by a number of authors such as \cite{Birch,Howe,BKP,KaplanPetrow2,KaplanPetrow,Lenstra,MP}. For example, in \cite[Theorem 1.2]{BKP} it was shown that as $p\to \infty$ (if $\nu=0$, we omit the subscript throughout, so, for example, we abbreviate $S_{0,m,M}$ as $S_{m,M}$)
\begin{equation}\label{eqn:evenEllipticpr}
	\frac{ S_{2k,m,M}(p)}{2^{2k}p^kS_{m,M}(p)}  = \frac{C_k}{2^{2k}}+O\left(p^{-\frac{1}{2}+\varepsilon}\right), \quad
	\frac{ S_{2k,m,M}(p^r)}{2^{2k}p^{rk} S_{m,M}(p^r)} = \frac{C_k}{2^{2k}}+O\left(p^{-1+\varepsilon}\right)\ (r\geq 2),
\end{equation}
where $C_k:=\frac{(2k)!}{k!(k+1)!}$ is the $k$-th \textit{Catalan number}. Since the main term in \eqref{eqn:evenEllipticpr} matches the $2k$-th moment of $\mu_{\operatorname{ST}}$ (see Lemma \ref{lem:SatoTateMoments}), it  remains to study the odd moments. In the case $M=1$, all of the odd moments vanish due to a symmetry which implies that for\footnote{Throughout we are making this assumption in order to avoid difficulties that arise from the fact that the short Weierstrass form is not general enough to include all elliptic curves in characteristic $2$ and $3$.} $p>3$ (see \cite[Theorem 3]{KaplanPetrow}) 
\begin{equation}\label{eqn:symmetry}
\sum_{\substack{E/\F_{p^r}\\ \tr(E)=-t}} \omega_E^{-1}=\sum_{\substack{E/\F_{p^r}\\ \tr(E)=t}} \omega_E^{-1}.
\end{equation}
For $M>1$, the symmetry giving \eqref{eqn:symmetry} is broken, but we prove in this paper that there is still a large amount of cancellation. We vary both the prime $p$ and the power $r$, getting equivalent cancellation in both cases. 
\begin{theorem}\label{thm:EllipticOdd}
Let $k\in\N$, $m\in\Z$, and $M\in\N$.
\begin{enumerate}[leftmargin=*]
\item[\rm (1)]
 For a fixed prime $p>3$, as $r\to\infty$ we have 
\[
S_{2k+1,m,M}(p^r)=O_{k,M,\varepsilon}\left(p^{(k+1+\varepsilon)r}\right).
\]
\item[\rm (2)] For $r\in\N$ fixed, as $p\to\infty$ we have 
\[
S_{2k+1,m,M}(p^r)=O_{k,M,\varepsilon}\left(p^{(k+1+\varepsilon)r}\right).
\]
\end{enumerate}
\end{theorem}
\begin{remarks}
\noindent

\begin{enumerate}[leftmargin=*]
\item Using \eqref{eqn:Sratio}, Theorem \ref{thm:EllipticOdd} implies that  
\begin{equation}\label{eqn:SratioOdd}
\frac{1}{S_{m,M}(p^r)}\sum_{\substack{E/\F_{p^r}\\ \tr(E)\equiv m\pmod{M}}} \frac{x_E^{2k+1}}{\omega_E}=O_{k,M,\varepsilon}\left(p^{\left(-\frac{1}{2}+\varepsilon\right)r}\right).
\end{equation}
Here we use the fact that 
\begin{equation}\label{eqn:SmMlower}
S_{m,M}(p^r)\gg_{M,\varepsilon} p^{(1-\varepsilon)r},
\end{equation}
 which follows by plugging \cite[Lemma 3.7]{BKP} and \cite[Lemma 4.2]{BKP} into \cite[Lemma 2.2]{BKP}.

\item  By Birch's result \cite{Birch} and its generalizations \cite{BrockGranville,KatzSarnak}, for any constant $0<c<2$ the set 
\[
	\#\left\{E/\F_{p^r}: |\tr(E)|>c p^{\frac{r}{2}}\right\}
\]
has positive density in $\{E/\F_{p^r}\}$ as $p+r\to\infty$. Hence if one takes the moments in \eqref{eqn:SmMkdef} with $\tr(E)$ replaced with $|\tr(E)|$, then 
\[
	\quad \sum_{\substack{E/\F_{p^r}\\ \tr(E)\equiv m\pmod{M}}} \frac{|\tr(E)|^{\nu}}{\omega_E}\geq \sum_{\substack{E/\F_{p^r}\\ \tr(E)\equiv m\pmod{M}\\ \tr(E)>cp^{\frac{r}{2}}}} \frac{\tr(E)^{\nu}}{\omega_E}\gg_{c,\nu} p^{\frac{r\nu}{2}}S_{m,M}\left(p^r\right)\gg_{M,\varepsilon} p^{\left(\frac{\nu}{2}+1-\varepsilon\right)r}.
\]
In the last inequality, we use the bound  \eqref{eqn:SmMlower} again. Taking $\nu=2k+1$, we conclude that if there would be no cancellation in $S_{2k+1,m,M}(p^r)$, then we would have 
\[
S_{2k+1,m,M}(p^r)\gg p^{r\left(k+\frac{3}{2}-\varepsilon\right)}.
\]
We hence see that the cancellation yields a power savings of (at least) $p^{(\frac12-\varepsilon)r}$.
\end{enumerate}
\end{remarks}

By work of Deuring \cite{Deuring} and Schoof \cite{Schoof}, the number of isomorphism classes of elliptic curves $E/\F_{p^r}$ whose trace of Frobenius equals $t$ is closely related to class numbers of imaginary quadratic fields. Our investigation of the moments related to the trace of Frobenius goes through moments associated to these class numbers, or, equivalently, the \begin{it}Hurwitz class numbers\end{it} $H(n)$ for classes of binary quadratic forms, which we next define. For $n\in\N$, we let $H(n)$ denote the number of $\SL_2(\Z)$-inequivalent classes of integral binary quadratic forms $[a,b,c]$ of discriminant $-n$, counted with multiplicity $\frac{1}{2}$ if $[a,b,c]$ is $\SL_2(\Z)$-equivalent to $[a,0,a]$ and multiplicity $\frac{1}{3}$ if $[a,b,c]$ is $\SL_2(\Z)$-equivalent to $[a,a,a]$. We furthermore set $H(0):=-\frac{1}{12}$ and $H(n):=0$ for $n<0$. For $\nu\in\N_0$, $M\in\N$, $m\in\Z$, and $n\in\N_0$, we define 
\begin{equation}\label{eqn:HmMkdef}
H_{\nu,m,M}(n):=\sum_{t\equiv m \pmod{M}} H \left(4n-t^2\right) t^{\nu}.
\end{equation}
Analogously to the remark after Theorem \ref{thm:EllipticOdd}, if one replaces $t$ with $|t|$ in \eqref{eqn:HmMkdef}, then the main contribution is expected to come from $t$ of size roughly $|t|\approx \sqrt{n}$. Hence it is natural to divide \eqref{eqn:HmMkdef} by the expected average absolute value (abbreviating $H_{m,M}(n):=H_{0,m,M}(n)$)
\[
	n^{\frac{\nu}{2}} \sum_{t\equiv m\pmod{M}} H \left(4n-t^2\right) = n^{\frac{\nu}{2}} H_{m,M}(n).
\]
Indeed, in \cite[Theorem 1.1]{BKP}, it was shown that for $k\in\N_0$
\[
	\frac{H_{2k,m,M}(n)}{n^kH_{m,M}(n)} =C_k+O_{k,M,\varepsilon}\left(n^{-\frac{1}{2}+\varepsilon}\right).
\]
Combining this with \cite[Lemma 3.7]{BKP}, one also obtains
\begin{equation}\label{eqn:HEvenMoment}
H_{2k,m,M}(n)=C_k n^kH_{m,M}(n) + O_{k,M,\varepsilon}\left(n^{k+\frac{1}{2}+\varepsilon}\right).
\end{equation}
 Since $t^{\nu}$ alternates in sign for odd $\nu$, one again expects cancellation. We show that this happens in the following theorem.
\begin{theorem}\label{thm:OddMoments}
For $k\in\N_0$, we have
\[
\frac{H_{2k+1,m,M}(n)}{n^{k+\frac{1}{2}} H_{m,M}(n)} = O_{k,M,\varepsilon}\left(n^{-\frac{1}{2}+\varepsilon}\right), \quad
H_{2k+1,m,M}(n)=O_{k,M,\varepsilon}\left(n^{k+1+\varepsilon}\right).
\]
\end{theorem}
\begin{remark}
Combining Theorem \ref{thm:OddMoments} with \eqref{eqn:HEvenMoment}, for $\nu\in\N_0$ we have 
\[
H_{\nu,m,M}(n)=\delta_{2\mid \nu} C_{\frac{\nu}{2}}  n^{\frac{\nu}{2}}H_{m,M}(n) + O_{\nu,M,\varepsilon}\left(n^{\frac{\nu+1}{2}+\varepsilon}\right),
\]
where $\delta_S:=1$ if a statement $S$ holds and $\delta_S:=0$ otherwise.
\end{remark}

The paper is organized as follows. In Section \ref{sec:prelim}, we give background information. In Section \ref{sec:OddMoments}, we evaluate the coefficients of the Rankin--Cohen brackets between the Hurwitz class number generating function and a weight $\frac{3}{2}$ unary theta function. The Rankin--Cohen brackets are then related to the odd  moments in Section \ref{sec:OddMomentsRankinCohen}, where Theorem \ref{thm:OddMoments} is proven. We show Theorem \ref{thm:EllipticOdd} in Section \ref{sec:elliptic}. Finally, we consider the applications to distributions of the trace of Frobenius in Section \ref{sec:SatoTate}, proving Theorem \ref{thm:SatoTateArithmetic} and Theorem \ref{thm:SatoTateArithmetic2}.

\section*{Acknowledgements}
The referees thank the anonymous referees for their detailed and helpful comments and suggestions. The first author has received funding from the European Research Council (ERC) under the European Union's Horizon 2020 research and innovation programme (grant agreement No. 101001179). The research of the second author was supported by a grant from the Research Grants Council of the Hong Kong SAR, China (project number HKU 17303618).  The research of third author was supported by ERCIM ‘Alain Bensoussan’ Fellowship Programme and the Science and Engineering Research Board (SRG/2023/000930).   The third author would like to thank The University of Hong Kong where he was a post-doctoral fellow,  where partially the  project was carried out, for providing excellent facilities and environment and the Institute of Mathematical Research at the University of Hong Kong for providing partial financial assistance to the post-doctoral fellowship the author was availing  during the course of this project.

\section{Preliminaries}\label{sec:prelim}
\subsection{Holomorphic and non-holomorphic modular forms}
We give a brief overview of the theory of modular forms here; for details, see e.g. \cite{Koblitz,OnoBook}. We assume throughout that $\Gamma\subseteq\Gamma_0(4)$ is a congruence subgroup containing $T:=\left(\begin{smallmatrix}1&1\\ 0 &1\end{smallmatrix}\right)$ and  $\kappa\in\frac{1}{2}\Z$. We let $\gamma=\left(\begin{smallmatrix}a&b\\c&d\end{smallmatrix}\right)\in\Gamma$ act on the complex \begin{it}upper half-plane\end{it} $\H:=\{\tau\in\C: \im(\tau)>0\}$ via \begin{it}fractional linear transformations\end{it} $\gamma \tau:=\frac{a\tau+b}{c\tau+d}$. For $d$ odd, set
\[
	\varepsilon_{d}:=\begin{cases} 1 &\text{if }d\equiv 1\pmod{4}\hspace{-1pt},\\ i&\text{if }d\equiv 3\pmod{4}\hspace{-1pt}.\end{cases}
\]
For $\gamma=\left(\begin{smallmatrix}a&b\\c&d\end{smallmatrix}\right)\in\Gamma$ 
and a function $F:\H\to \C$, we define the  weight $\kappa$ \begin{it}slash operator\end{it} by 
\[
F\big|_{\kappa}\gamma(\tau):= \left( \frac cd \right)^{2\kappa} \varepsilon_d^{2\kappa}(c\tau+d)^{-\kappa} F(\gamma\tau),
\]
where we denote the extended Legendre symbol by $(\frac{\cdot}{\cdot})$. 
Suppose that $\Gamma_1(N)\subseteq \Gamma$ for some $N\in\N$ and $\chi$ is a Dirichlet character $\hspace{-0.2cm}\Pmod N$. A function $F:\H\to\C$ satisfies \begin{it}modularity of weight $\kappa$ on $\Gamma$ \end{it}\textit{with character $\chi$} if for every $\gamma=\left(\begin{smallmatrix}a&b\\ c&d\end{smallmatrix}\right)\in\Gamma$ we have 
\[
F|_{\kappa}\gamma  = \chi(d) F.
\]
We call $F$ a \begin{it}(holomorphic) modular form\end{it} if $F$ is holomorphic on $\H$ and $F(\tau)$ grows at most polynomially in $v$ as $\tau=u+iv\to \Q\cup\{i\infty\}$. We say that $F:\H\to\C$ is an \begin{it}almost holomorphic modular form\end{it} if $F$ satisfies modularity of weight $\kappa$ on $\Gamma$ and there exist holomorphic functions $F_j$ ($0\leq j\leq \ell$) such that $F(\tau)=\sum_{j=0}^{\ell} F_j(\tau) v^{-j}$. We call $F_0$ a \begin{it}quasimodular form\end{it}.
\subsection{Rankin--Cohen brackets}
For $F_1,F_2$ transforming like modular forms of weight $\kappa_1,\kappa_2 \in \frac{1}{2}\Z$, respectively, define  for $k \in \N_0$ the $k$-th {\it Rankin-Cohen bracket} 
\begin{equation*}
[F_1,F_2]_k := \frac{1}{(2\pi i)^k}\sum_{j=0}^{k} (-1)^j \binom{\kappa_1 + k -1}{k-j} \binom{\kappa_2 + k -1}{j} F_1^{(j)} F_2^{(k-j)}
\end{equation*}
with  $\binom{\alpha}{j}:=\frac{\Gamma(\alpha+1)}{j!\Gamma(\alpha-j+1)}$. Here the \begin{it}gamma function\end{it} is defined by $\Gamma(s):=\int_{0}^{\infty} t^{s-1}e^{-t}dt$ for $\re(s)>0$ (and via meromorphic continuation for $s\in\C$) and we use the abbreviation $F^{(j)}(\tau):=\frac{\partial^j}{\partial \tau^j} F(\tau)$. Cohen showed in \cite[Theorem 7.1 (a)]{Cohen} that $[F_1,F_2]_k$ transforms like a modular form of weight $\kappa_1+\kappa_2+2k$.

\subsection{Holomorphic projection}
Suppose that we have a (not necessarily holomorphic) translation-invariant (i.e., invariant under $\tau\mapsto \tau+1$) function  $F(\tau) = \sum_{n \in \Z} c_{F,v}(n) q^{n}$, with $c_{F,v}(n)$ such that the series converges absolutely for all $\tau\in\H$. In \cite{Sturm}, Sturm observed that if $F$ satisfies modularity of weight $\kappa\geq 2$ on some congruence subgroup $\Gamma\subseteq\SL_2(\Z)$ and the {\it Petersson inner product}
\[
\left<F,g\right>:=\frac{1}{\left[\SL_2(\Z):\Gamma\right]} \int_{\Gamma\backslash\H} F(\tau)\overline{g(\tau)} v^{k} \frac{dudv}{v^2}
\]
 exists for every cusp form $g$ of weight $\kappa$ on $\Gamma$, then there exists a unique cusp form $f$ (in the same space) such that $\left<F,g\right>=\left<f,g\right>$ for every $g$. Sturm called this cusp form $f$ the {\it holomorphic projection} of $F$ and evaluated its Fourier coefficients as a certain integral. Gross and Zagier (see \cite[Proposition 5.1, p. 288]{GrossZagier}) instead used Sturm's evaluation via an integral as the definition of the holomorphic projection operator, which extends to a more general setting. Following the construction of \cite{GrossZagier}, we now relax the condition on $F$, dropping the assumption that it is modular and only assuming translation-invariance. If $c_{F,v}(n)\ll_{F,n} v^{2-\kappa}$ as $v\to 0^+$ for all $n\in\N$, then for $n \in \N$ we define 
\begin{align*}
c_{F}(n):= \frac{(4 \pi n)^{\kappa-1}}{\Gamma(\kappa-1)} \int_0^{\infty} c_{F,v} (n) v ^{\kappa-2} e^{-4 \pi n v} d v.
\end{align*}
If there exists a constant $c_F(0)\in\C$ for which $F(\tau)-c_F(0)$ decays as $\tau\to i\infty$, and if a similar condition holds as $\tau\to \varrho$ for any $\varrho\in\Q$, we define (see \cite{MOR} for the statement written in this generality) the \begin{it}holomorphic projection\end{it} of $F$ to be 
\begin{equation*}
	\pi_{\text{hol}} (F) (\tau) := c_F(\kappa,0) + \sum_{n =1}^\infty c_{F}(n) q^{n}.
\end{equation*}
Although we may use holomorphic projection with any $\kappa$ satisfying the growth conditions, one usually applies the holomorphic projection operator to a function $F$ satisfying modularity of weight $\kappa$, but then splits $F$ into a sum $F=\sum_{j=1}^{\ell} F_j$ for certain appropriate functions $F_1,\dots,F_{\ell}$, so the weight $\kappa$ is clear from the context. Using this more general definition above allows one to use the operator for the individual summands $F_j$ even though they may not be modular.

Mertens \cite{Me} studied the holomorphic projection of the Rankin--Cohen brackets between holomorphic modular forms and non-holomorphic modular forms with Fourier expansions of the type 
\begin{equation}\label{eqn:harmonicFourier}
F(\tau)=F^+(\tau)+F^{-}(\tau)
\end{equation}
with (for $q:=e^{2\pi i \tau }$) 
\[
F^+(\tau)=\sum_{n\gg -\infty} c_F^+(n) q^n,\qquad
F^-(\tau)=c_F^-(0) v^{1-\kappa} + \sum_{n\le-1} c_F^-(n) \Gamma\left(1-\kappa,4\pi |n| v\right) q^{n}.
\]
Here for $x\in\R_{>0}$ and $\alpha\in\C$, denote by $\Gamma(\alpha,x)=:\int_x^{\infty} t^{\alpha-1}e^{-t}dt$ the {\it incomplete gamma function}.
Set 
\[
F_0^-(\tau):=F^-(\tau)-c_F^-(0) v^{1-\kappa}. 
\]
We require the contribution to this holomorphic projection coming from the non-holomorphic part, which was given in \cite[Lemma 4.4 and Theorem 4.6]{Me}. To state it, for $a\in\N$ and $b\in\R$ we define the two-variable polynomials 
\[
P_{a,b}(X,Y):=\sum_{j=0}^{a-2}\binom{j+b-2}{j} X^j(X+Y)^{a-j-2}. 
\]
Moreover, for $\kappa_1,\kappa_2\in\R\setminus\Z$ and $k\in\N_0$ with $2k-2\geq \kappa_1+\kappa_2\in\Z$, we define 
\[
\alpha_{\kappa_1,\kappa_2,k}:= \frac{1}{(\kappa_1+\kappa_2+2k-2)!(\kappa_1-1)}  \sum_{\mu=0}^k \frac{\Gamma(2-\kappa_1)\Gamma(\kappa_2+2k-\mu)}{\Gamma(2-\kappa_1-\mu)}\binom{\kappa_1+k-1}{k-\mu} \binom{\kappa_2+k-1}{\mu}.
\]

\begin{lemma}\label{lem:holprojRankinCohen} 
Let $\kappa_1,\kappa_2\in \N_0+\frac{1}{2}$ and $k\in\N_0$ with $2k-2\geq \kappa_1+\kappa_2\in\Z$. Suppose that $F$ satisfies modularity of weight $\kappa_1$ and has a Fourier expansion of the type \eqref{eqn:harmonicFourier} that grows at most polynomially towards the cusps and that $g$ is a holomorphic modular form of weight $\kappa_2$.
Applying holomorphic projection in weight $\kappa_1+\kappa_2+2k$ to $\left[F,g\right]_k$ and splitting $F$ as in \eqref{eqn:harmonicFourier}, we have the following.
\begin{enumerate}[leftmargin=*,label={\rm(\arabic*)}]
\item We have 
\[
\frac{(4\pi)^{1-\kappa_1}}{\kappa_1-1} \pi_{\operatorname{hol}}\left(\left[v^{1-\kappa_1},g\right]_{k}\right) =\alpha_{\kappa_1,\kappa_2,k} \sum_{n=0}^{\infty} n^{\kappa_1+k-1} c_g(n) q^n.
\]
\item We have  
\[
\pi_{\operatorname{hol}}\left(\left[F_0^{-},g\right]_{k}\right)(\tau) = \sum_{n=1}^{\infty}b(n) q^n,
\]
where 
\begin{multline*}
\hspace{.35cm}b(n)=-\Gamma(1-\kappa_1)\sum_{\substack{j,\ell\in\N\\  j-\ell=n}}\sum_{\mu=0}^{k} \binom{\kappa_1+k-1}{k-\mu}\binom{\kappa_2+k-1}{\mu} j^{k-\mu} c_g(j) \frac{c_{F}^-(-\ell)}{\ell^{\kappa_1-1}}\\
\times\left(j^{\mu-2k-\kappa_2+1}P_{\kappa_1+\kappa_2+2k,2-\kappa_1-\mu}(n,\ell)-\ell^{\kappa_1+\mu-1}\right).
\end{multline*}

\end{enumerate}
\end{lemma}

\subsection{Examples of modular objects}\label{sec:prelimexamples}
We require a few specific modular objects. We first define 
\[
\theta_{\nu,m,M}(\tau):=\sum_{n\equiv m\pmod{M}} n^{\nu} q^{n^2},
\]
Setting $\Gamma_{N,M}:=\Gamma_0(N)\cap\Gamma_1(M)$ for $M\mid N$, this is a quasimodular form of weight  $\nu+\frac{1}{2}$ on $\Gamma_{4M^2,M}$ (see \cite[Proposition 2.1]{Shimura}). It is moreover a holomorphic modular form for $\nu=0$ and a cusp form for $\nu=1$. Let 
\[
\mathcal{H}(\tau):=\sum_{n\in\Z} H(n) q^n
\]
be the generating function for the Hurwitz class numbers. Its modular properties follow by \cite[Chapter 2, Theorem 2]{HZ}. 
\begin{theorem}\label{thm:Hcomplete}
The function 
\[
\widehat{\mathcal{H}}(\tau):=\mathcal{H}(\tau)+\frac{1}{8\pi \sqrt{v}}+ \frac{1}{4\sqrt{\pi}}\sum_{n=1}^\infty n\Gamma\left(-\frac12, 4\pi n^2 v\right)q^{-n^2}
\]
is a non-holomorphic modular form of weight $\frac{3}{2}$ on $\Gamma_0(4)$ which grows at most polynomially towards all cusps and has an expansion of the type \eqref{eqn:harmonicFourier}.
\end{theorem}
\subsection{Equidistribution and the method of moments}\label{sec:momentmethod}

\begin{definition}
Let $f$ be a real-valued continuous function with compact support. We define a measure $\mu_f$ on Lebesgue measurable sets $X$ by 
\[
\mu_f(X):=\int_{X} f(x)dx.
\]
In the special case $f=f_{\operatorname{ST}}$, we write $\mu_{\operatorname{ST}}=\mu_{f_{\operatorname{ST}}}$ for ease of notation. We say that a sequence of subsets $\mathcal{S}_j\subseteq \R$ is \begin{it}equidistributed with respect to $\mu_f$\end{it} if for every $a\leq b$ we have 
\[
\lim_{j\to\infty} \Pr\left(a\leq x\leq b: x\in \mathcal{S}_j\right) = \int_{a}^b f(x)dx=\mu_f([a,b]).
\]
We investigate $\mu_{f}$ through the
 \begin{it}$\nu$-th moment of $f$ at $0$\end{it}, 
which is defined by 
\[
\mu_{f,\nu}:=\int_{-\infty}^\infty x^{\nu} f(x) dx.
\]
\end{definition}
In the special case $f=f_{\operatorname{ST}}$, we write $\mu_{\operatorname{ST},\nu}$. Suppose now that $f$ is a probability density function and let $X$ be a random variable corresponding to $f$; i.e., $X$ randomly assigns a value in $\R$ such that for $a<b$ 
\[
\operatorname{Pr}\left(a\leq X\leq b\right)=\int_{a}^{b} f(x)dx.
\]
Recall that for a function $g:\R\to\R$, the \begin{it}expected value of $g(X)$\end{it} is 
\[
E\left(g(X)\right):=\int_{-\infty}^{\infty} g(x)f(x)dx.
\]
In this case, we see that the $\nu$-th moment is
\[
E\left(X^{\nu}\right)=\mu_{f,\nu}.
\]
Under certain mild conditions, the moments of a distribution uniquely determine the distribution (see~\cite[Theorem 30.1 and Theorem 30.2]{Bil}). In order to prove Theorem \ref{thm:SatoTateArithmetic}, we require the moments of the Sato--Tate measure, which are well-known (for example, see the introduction of \cite{Birch}). 
\begin{lemma}\label{lem:SatoTateMoments}
For $k\in\N_0$, the $2k$-th moment of the Sato--Tate distribution is
\[
\mu_{\operatorname{ST},2k}= \frac{C_k}{2^{2k}},
\]
while $\mu_{\operatorname{ST},2k+1}=0$.
\end{lemma}

\subsection{Weighted distributions and equidistribution}\label{sec:weightedist}

By determining when two of the elliptic curves $E_{\alpha,\beta}$ are isomorphic (see \cite[discussion after (1)]{Birch}), Birch's result and its generalizations may be naturally described using a weighted probability on $\mathcal{E}(p^r)$. For a statement $\mathcal{S}(E)$ which is true or false for each $E$ in a subset $\mathcal{E}$ of elliptic curves over $\F_{p^r}$, set $\mathcal{E}_{\mathcal{S}}:=\{E\in \mathcal{E}: \mathcal{S}(E)\text{ holds}\}$ and
\begin{equation*}
\operatorname{Pr}_{\Aut}\left(\mathcal{S}(E): E\in \mathcal{E}\right):=\frac{{\displaystyle\sum_{E\in\mathcal{E}_{\mathcal{S}}}} \omega_E^{-1} }{{\displaystyle{\sum_{E\in\mathcal{E}}}}\omega_E^{-1}}.
\end{equation*}
Using Birch's counting argument in \cite[discussion after (1)]{Birch} for the number of $E_{\alpha,\beta}$ in each isomorphism class, Birch's result may be written as 
\[
\lim_{p\to\infty}\operatorname{Pr}_{\Aut}\left(a\leq x_E\leq b: E\in \mathcal{E}(p)\right) = \mu_{\ST}([a,b]).
\]
It is well-known that the discrepancy between the weighted and unweighted probabilities vanishes as $p\to\infty$. The corresponding relationship for the weighted and unweighted moments of traces in arithmetic progressions is given in the following lemma.

\begin{lemma}\label{lem:weightedunweighted}
	Let $m\in\Z$ and $M,r\in\N$. Then as $p\to\infty$ we have 
	\begin{equation*}
		S_{\nu,m,M}(p^r)=  2^{\nu-1} p^{\frac{r\nu}{2}}\sum_{E\in\mathcal{E}_{m,M}(p^r)} x_E^{\nu} + O_{\nu}\left(p^{\frac{r\nu}{2}} \right).
	\end{equation*}
\end{lemma}

\begin{proof}
First note that for $\nu\in\N_0$ we have 
\begin{equation}\label{first}
S_{\nu,m,M}(p^r) = 2^{\nu-1}p^{\frac{r\nu}{2}} \sum_{E\in\mathcal{E}_{m,M}(p^r)} x_E^{\nu} + 2^{\nu}p^{\frac{r\nu}{2}} \sum_{\substack{E\in \mathcal{E}_{m,M}(p^r)\\ \omega_{E}\neq 2}} \left(\frac{1}{\omega_E}-\frac{1}{2}\right) x_{E}^{\nu}.
\end{equation}
It is well-known that the number of $E/\F_{p^r}$ with $\omega_E\neq 2$ is bounded by an absolute constant (see \cite[Chapter III and Chapter X]{Silverman} for further details). Hence, as $p\to\infty$, by \eqref{first}
\[
\left|\sum_{\substack{E\in \mathcal{E}_{m,M}(p^r)\\ \omega_{E}\neq 2}} \left(\frac{1}{\omega_E}-\frac{1}{2}\right) x_{E}^{\nu}\right|\leq \frac{1}{2}\#\{E/\F_{p^r}: \omega_E\neq 2\}\ll 1,
\]
yielding the lemma. 
\end{proof}

\section{Evaluation of the Rankin--Cohen brackets}\label{sec:OddMoments}

 As in \cite[Lemma 2.6 (3)]{BKP} (see also \cite[(4.6)]{Me}, where the Fourier expansion differs slightly), we have 
\begin{equation}\label{eqn:holprojrelate}
\left[\mathcal{H},\theta_{1,m,M}\right]_{k}=\pi_{\operatorname{hol}}\left(\left[\widehat{\mathcal{H}},\theta_{1,m,M}\right]_{k}\right)-\pi_{\operatorname{hol}}\left(\left[  \mathcal{H}^{-},\theta_{1,m,M}\right]_{k}\right).
\end{equation}
Writing 
\begin{equation}\label{eqn:HThetaFourier}
\left[\mathcal{H},\theta_{1,m,M}\right]_{k}(\tau)=\sum_{n=0}^\infty c_{k,m,M}(n)q^n,
\end{equation}
the idea is to obtain asymptotic behavior $c_{k,m,M}(n)$ by bounding the coefficients of the individual terms on the right-hand side of \eqref{eqn:holprojrelate}. We begin by considering the first term on the right-hand side of \eqref{eqn:holprojrelate}. We recall its modular properties, which leads to a bound by Deligne \cite{Deligne}.


\begin{lemma}\label{lem:piholodd}
For $k\in\N_0$ we have 
\[
\pi_{\operatorname{hol}}\left(\left[\widehat{\mathcal{H}},\theta_{1,m,M}\right]_{k}\right)\in S_{2k+3}\left(\Gamma_{4M^2,M}\right).
\]
Hence the $n$-th coefficient of $\pi_{\operatorname{hol}}([\widehat{\mathcal{H}},\theta_{1,m,M}]_{k})$ is $O(n^{k+1+\varepsilon})$.
\end{lemma}
\begin{proof}
From \cite[Proposition 4.2]{Me}, $\pi_{\operatorname{hol}}([\widehat{\mathcal{H}},\theta_{1,m,M}]_{k})$ is a modular form of weight $2k+3$ on $\Gamma_{4M^2,M}$. To see that it is a cusp form, we follow the same argument as given in \cite{Me} between \cite[(7.2) and (7.3)]{Me}. Namely, since $\theta_{1,m,M}$ is a cusp form and holomorphic projection and the Rankin--Cohen bracket both commute with slashing, for $\gamma\in\SL_2(\Z)$ we have 
\[
\pi_{\operatorname{hol}}\left(\left[\widehat{\mathcal{H}},\theta_{1,m,M}\right]_{k}\right)\big|_{2k+3} \gamma = \pi_{\operatorname{hol}}\left(\left[\widehat{\mathcal{H}}\big|_{\frac{3}{2}}\gamma,\theta_{1,m,M}\big|_{\frac{3}{2}}\gamma\right]_{k}\right),
\]
and the vanishing of the constant term of $\theta_{1,m,M}|_{\frac 32}\gamma$ implies that the holomorphic projection vanishes at every cusp, and is hence a cusp form. 

For the second claim, we use Deligne's bound \cite{Deligne} for cusp forms.
\end{proof}

Using Lemma \ref{lem:piholodd}, we next relate $c_{k,m,M}(n)$ to
\begin{equation*}
F_{k,t}(s):=2^{-2k} \frac{2k+1}{2k+2}\binom{2k}{k} (t-s)^{2k+2}.
\end{equation*}
To state the result, we use $\sum_{s,t}^*$ to denote the sum over $s$ and $t$ where the term with $s=0$ is weighted by $\frac{1}{2}$. 
\begin{lemma}\label{lem:c-F}
We have\footnote{Here the notation $\displaystyle\sum_\pm$ means that we sum over the choice $+$ as well as the choice $-$.} 
\begin{equation*}
c_{k,m,M}(n)=-\frac{1}{2}\sum_{\pm} \pm \sideset{}{^*}{\sum}_{\substack{s,t\in\N_0\\ t^2-s^2=n\\ t\equiv \pm m\pmod{M}}} F_{k,t}(s)+O_{k,M}\left(n^{k+1+\varepsilon}\right).
\end{equation*}
\end{lemma}
\begin{proof}
Plugging Lemma \ref{lem:piholodd} into \eqref{eqn:holprojrelate}, it remains to show that the $n$-th Fourier coefficient of $\pi_{\operatorname{hol}}\left(\left[  \mathcal{H}^{-},\theta_{1,m,M}\right]_{k}\right)$ is 
\[
-\frac{1}{2}\sum_{\pm}\pm \sideset{}{^*}{\sum}_{\substack{s,t\in\N_0\\ t^2-s^2=n\\ t\equiv \pm m\pmod{M}}} F_{k,t}(s)+O_{k,M}\left(n^{k+1+\varepsilon}\right).
\]
To do so, we separately consider  the contribution from the (non-holomorphic) constant term of $\mathcal{H}^-$ and the remaining terms, starting with the constant term. Noting that the $n$-th Fourier coefficient of $\theta_{1,m,M}$ is
\[
\begin{cases} 
t&\text{if }n=t^2\text{ and }t\equiv m\pmod{M}\text{ with } -t\not\equiv m\pmod{M},\\ 
0&\text{otherwise},
\end{cases}
\]
we plug in Lemma \ref{lem:holprojRankinCohen} (1) with $\kappa_1=\kappa_2=\frac{3}{2}$ and $g=\theta_{1,m,M}$ to evaluate
\[
\pi_{\operatorname{hol}}\left(\left[\frac{1}{\sqrt{v}}, \theta_{1,m,M}\right]_{k}\right)=\alpha_{\frac{3}{2},\frac{3}{2},k}\sum_{\pm}\pm\sum_{\substack{s,t\in\N_0\\ t\equiv \pm m\pmod{M}}}t^{2k+2} q^{t^2}.
\]
Hence the $t^2$-th coefficient is $O(n^{k+1})$.

It remains to evaluate the contribution from $\mathcal{H}_0^-(\tau)=\mathcal{H}^-(\tau)-\frac{1}{8\pi\sqrt{v}}$. We use Lemma \ref{lem:holprojRankinCohen} (2) with $\kappa_1=\kappa_2=\frac{3}{2}$, $F=\widehat{\mathcal{H}}$, and $g=\theta_{1,m,M}$. Writing $j=t^2$ and $\ell=s^2$ in the sum defining $b(n)$ in Lemma \ref{lem:holprojRankinCohen} (2), the sum runs over $t^2-s^2=n$ with $t\in\N$ satisfying $t\equiv \pm m\pmod{M}$ and $s\in\N$ (i.e., $n=(t+s)(t-s)$ and using $\Gamma(-\frac12)=-2\sqrt{\pi}$, the $n$-th Fourier coefficient of $\pi_{\operatorname{hol}}([\widehat{\mathcal{H}}_0^-,\theta_{1,m,M}]_k)$ equals
\begin{multline}\label{eqn:c-}
2\sqrt{\pi}\sum_{\pm}\pm \sum_{\substack{s,t\in\N_0\\ t^2-s^2=n\\t\equiv\pm m\pmod{M}}} \sum_{\mu=0}^{k} \binom{k+\frac12}{k-\mu}\binom{k+\frac12}{\mu} t^{2k-2\mu+1} \frac{1}{4\sqrt{\pi}} \left(t^{2\mu-4k-1}P_{3+2k,\frac{1}{2}-\mu}\left(n,s^2\right)-s^{2\mu+1}\right)\\
\hspace{-.1cm}= \frac12\sum_{\pm}\pm \sum_{\substack{s,t\in\N_0\\ t^2-s^2=n\\t\equiv\pm m\pmod{M}}} \sum_{\mu=0}^{k} \binom{k+\frac12}{k-\mu}\binom{k+\frac12}{\mu} t^{2k-2\mu+1} \left(t^{2\mu-4k-1} P_{2k+3,\frac{1}{2}-\mu} \left(n,s^2\right) -s^{2\mu+1}\right).
\end{multline}
By \cite[Proposition 4.2]{OnoSaadSaikia}, we have 
\[
 \sum_{\mu=0}^{k}\binom{k+\frac{1}{2}}{k-\mu}\binom{k+\frac{1}{2}}{\mu} t^{2k-2\mu+1} \left(t^{2\mu-4k-1}P_{2k+3,\frac{1}{2}-\mu}\left(n,s^2\right)-s^{2\mu+1}\right)=F_{k,t}(s),
\]
and thus \eqref{eqn:c-} simplifies to give the claim.
\end{proof}

We next use Lemma \ref{lem:c-F} to rewrite the asymptotics at $c_{\ell,m,M}(n)$ in terms of sums of divisors.
\begin{proposition}\label{prop:cmMkeval}
Let $m\in\Z$, $M\in\N$, and $k\in\N_0$. Then, as $n\to\infty$, for any  $\varepsilon>0$ we have
\[
c_{k,m,M}(n)=-2^{-2k-1} \frac{2k+1}{2k+2}\binom{2k}{k} \sum_{\pm} \pm \sideset{}{^*}{\sum}_{\substack{d\mid n,\ d^2\leq n\\ d+\frac{n}{d}\equiv \pm 2m\pmod{2M}}} d^{2k+2}+O_{k,M,\varepsilon}\left(n^{k+1+\varepsilon}\right),
\]
where $\sum\!{\vphantom{\sum}}^*$ means the term with $d=\sqrt{n}$ (if it occurs) is counted with multiplicity $\frac{1}{2}$. 

\end{proposition}
\begin{proof}
By Lemma \ref{lem:c-F}, we have
\begin{equation*}
c_{k,m,M}(n)=-2^{-2k-1} \frac{2k+1}{2k+2}\binom{2k}{k}\sum_{\pm}\pm \sideset{}{^*}{\sum}_{\substack{s,t\in\N_0\\ t^2-s^2=n\\ t\equiv \pm m\pmod{M}}} (t-s)^{2k+2}+O_{k,M}\left(n^{k+1+\varepsilon}\right).
\end{equation*}
We  now write $t-s=d$ and $t+s=\frac{n}{d}\geq d$ (using that $s\geq 0$).  Note that there is a bijection between pairs $(t,s)$ and $d\mid n$ with $d^2\leq  n$ satisfying the congruence $d+\frac{n}{d}\equiv \pm 2m\pmod{2M}$ because $2t=d+\frac{n}{d}$ and $2s=\frac{n}{d}-d$ are uniquely determined from $d$. Using $d+\frac{n}{d}=2t$, this gives the claim.
\end{proof}

\section{Proof of Theorem \ref{thm:OddMoments}}\label{sec:OddMomentsRankinCohen}

Recalling \eqref{eqn:HThetaFourier}, we write 
\begin{equation*}
\left[\mathcal{H},\theta_{1,m,M}\right]_{k}\big| U_4(\tau)=\frac{1}{(2\pi i)^k}\sum_{j=0}^{k} d_{k,j} \left(\mathcal{H}^{(j)} \theta_{1,m,M}^{(k-j)}\right)  \big|U_4 (\tau)= \sum_{n=0}^\infty c_{k,m,M}(4n) q^n,
\end{equation*}
where
\begin{equation*}
d_{k,j}:= (-1)^j\binom{k+\frac{1}{2}}{j}  \binom{k+\frac{1}{2}}{k-j}
\end{equation*}
 are the coefficients from the Rankin--Cohen brackets. We next solve for $H_{2\ell+1,m,M}(n)$ in terms of $c_{k,m,M}(4n)$ and $H_{2k-2\ell+1,m,M}(n)$ with $\ell>0$. For this, define 
\[
\mathcal{C}_k:=\sum_{j=0}^k (-1)^j d_{k,j}.
\]
A straightforward calculation shows the following.
\begin{lemma}\label{lem:HGtilde}
We have 
\[
H_{2k+1,m,M}(n)=\frac{1}{\mathcal{C}_k}c_{k,m,M}(4n) -\frac{1}{\mathcal{C}_k}\sum_{j=1}^{k} d_{k,j} \sum_{\ell=1}^{j}(-1)^{j+\ell} \binom{j}{\ell} (4n)^{\ell}  H_{2k-2\ell+1,m,M}(n).
\]
\end{lemma}

We next bound $c_{k,m,M}(4n)$.
\begin{lemma}\label{lem:Gtildesmall}
Let $k\in\N_0$, $m\in\Z$, and $M\in\N$. Then we have,  as $n\to\infty$,
\[
c_{k,m,M}(4n)\ll_{k,M,\varepsilon} n^{k+1+\varepsilon}. 
\]
\end{lemma}
\begin{proof}
By Proposition \ref{prop:cmMkeval}, we have
\[
c_{k,m,M}(4n)=-2^{-2k-1} \frac{2k+1}{2k+2}\binom{2k}{k}\sum_{\pm} \pm \sideset{}{^*}{\sum}_{\substack{d\mid 4n,\ d^2\leq 4n\\ d+\frac{4n}{d}\equiv \pm 2m\pmod{2M}}} d^{2k+2}+O_{k,M,\varepsilon}\left(n^{k+1+\varepsilon}\right).
\]
Since $d^2\leq 4n$, we have
\[
d^{2k+2}\leq 4^{k+1} n^{k+1},
\]
so (noting that $d\mid 4n$ and $\sigma_0(n)\ll_{\varepsilon} n^{\varepsilon}$ as $n \to \infty$, where $\sigma_r(n):=\sum_{d\mid n} d^r$)
\[
c_{k,m,M}(4n)\ll_{k,M,\varepsilon} n^{k+1} \sigma_0(4n) + n^{k+1+\varepsilon}\ll_{k,M,\varepsilon} n^{k+1+\varepsilon}.\qedhere
\]
\end{proof}

We are now ready to prove Theorem \ref{thm:OddMoments}.

\begin{proof}[Proof of Theorem \ref{thm:OddMoments}]
By \cite[Lemma 3.7]{BKP}, we have 
\[
H_{m,M}(n)\gg_{\varepsilon} n^{1-\varepsilon}. 
\]
Thus if
\begin{equation}\label{eqn:OddUpper}
H_{2k+1,m,M}(n)=O_{k,M,\varepsilon}\left(n^{k+1+\varepsilon}\right),
\end{equation}
then 
\[
\frac{H_{2k+1,m,M}(n)}{n^{k+\frac{1}{2}}H_{m,M}(n)}=O_{k,M,\varepsilon}\left(n^{-\frac{1}{2}+\varepsilon}\right).
\]
It hence suffices to show \eqref{eqn:OddUpper}. We prove \eqref{eqn:OddUpper} by induction. For $k=0$, Lemma \ref{lem:HGtilde} followed by Lemma \ref{lem:Gtildesmall} implies that
\[
H_{1,m,M}(n)=\frac{1}{\mathcal{C}_0}c_{0,m,M}(4n)\ll_{M,\varepsilon} n^{1+\varepsilon}.
\]
This gives the claim for $k=0$. We next assume the claim for $0\leq j<k$ and use Lemma \ref{lem:HGtilde}. Together with Lemma \ref{lem:Gtildesmall} and the inductive hypothesis, this gives the claim.
\end{proof}

\section{Proof of Theorem \ref{thm:EllipticOdd}}\label{sec:elliptic}

We first recall that by \cite[Lemma 2.2]{BKP}, we have 
\begin{equation}\label{eqn:SHErel}
2S_{\nu,m,M}(p^r) =H_{\nu,m,M}(p^r) + E_{\nu,m,M}(p^r), 
\end{equation}
where 
\begin{multline*}
E_{\nu,m,M}\left(p^r\right):=\delta_{M\mid m}\delta_{\nu=0}\delta_{2\nmid r} H(4p)+ \delta_{M\mid m}\delta_{\nu=0}\delta_{2\mid r} \frac{1}{2}\left(1-\left(\frac{-1}{p}\right)\right)\\
 + \frac{1}{3}\left(1-\left(\frac{-3}{p}\right)\right)p^{\frac{\nu r}{2}}\varrho_{\nu,m,M}(p^r)+ \frac 13 (p-1) 2^{\nu - 2}p^{\frac{\nu r}{2}}\sigma_{\nu,m,M}(p^r)
\end{multline*}
with 
\begin{align*}
\varrho_{\nu,m,M}(p^r)&:=\displaystyle\sum_{\substack{t\equiv m\pmod{M}\\ t^2=p^r}}\sgn(t)^{\nu},&\sigma_{\nu,m,M}(p^r)&:=\displaystyle\sum_{\substack{t\equiv m\pmod{M}\\ t^2=4p^r}}\sgn(t)^{\nu}.
\end{align*}
Analogously to \cite[Lemma 4.2]{BKP}, we require a bound for $E_{2k+1,m,M}(p^r)$ with $k\in\N_0$. 
\begin{lemma}\label{lem:Eboundodd}
For $k\in\N_0$ and $r\in\N$, we have 
\[
E_{2k+1,m,M}(p^r)=\frac{2^{2k-1}}{3}p^{\left(k+\frac{1}{2}\right) r+1}\sigma_{2k+1,m,M}(p^r)+O_k\left(p^{\left(k+\frac{1}{2}\right)r}\right).
\]
 Moreover, 
\[
E_{2k+1,m,M}(p^r)=O_k\left(p^{\left(k+\frac{1}{2}\left(1+\delta_{2|r}\right)\right)r}\right).
\]
\end{lemma}

\begin{proof}
The first two terms in the definition of $E_{2k+1,m,M}(p^r)$ vanish because $2k+1>0$ by assumption. We then note that 
\begin{align}
\label{eqn:varrhobound}\left|\varrho_{\nu,m,M}(p^r)\right|&\leq \displaystyle\sum_{\substack{t\equiv m\pmod{M}\\ t^2=p^r}}\left|\sgn(t)^{\nu}\right|\leq \displaystyle\sum_{\substack{t\equiv m\pmod{M}\\ t^2=p^r}}1\leq 2,\\
\label{eqn:sigmabound}\left|\sigma_{\nu,m,M}(p^r)\right|&\leq\displaystyle\sum_{\substack{t\equiv m\pmod{M}\\ t^2=4p^r}}\left|\sgn(t)^{\nu}\right|= \displaystyle\sum_{\substack{t\equiv m\pmod{M}\\ t^2=4p^r}}1\leq 2.
\end{align}
From \eqref{eqn:varrhobound}, the third term in the definition of $E_{2k+1,m,M}(p^r)$ is $O(p^{(k+\frac{1}{2})r})$. For the final term, we split the factor $p-1$ into two terms, with $p$ giving the main term and $-1$ giving $O_k(p^{(k+\frac{1}{2})r})$ by \eqref{eqn:sigmabound}. This yields the first claim.

If $r$ is odd, then $t^2=4p^r$ is not solvable, thus $\sigma_{2k+1,m,M}(p^r)=0$, giving the second claim. For $r$ even, we have $r\geq 2$, so by \eqref{eqn:sigmabound}
\[
\frac{2^{2k-1}}{3}p^{\left(k+\frac{1}{2}\right) r+1}\sigma_{2k+1,m,M}(p^r)\ll_k p^{\left(k+\frac{1}{2}\right) r+\frac{r}{2}}=O_{k}\left(p^{(k+1)r}\right),
\]
as claimed.
\end{proof}

We are now ready to prove Theorem \ref{thm:EllipticOdd}.
\begin{proof}[Proof of Theorem \ref{thm:EllipticOdd}]
Plugging Lemma \ref{lem:Eboundodd} into \eqref{eqn:SHErel}, we have 
\[
S_{2k+1,m,M}(p^r)=\frac{1}{2} H_{2k+1,m,M}(p^r) +O_k\left(p^{(k+1)r}\right).
\]
We then plug in Theorem \ref{thm:OddMoments} to obtain 
\[
S_{2k+1,m,M}(p^r)=O_{k,M,\varepsilon}\left(p^{(k+1+\varepsilon)r}\right). \qedhere
\]
\end{proof}

\section{Proof of Theorems \ref{thm:SatoTateArithmetic} and \ref{thm:SatoTateArithmetic2}}\label{sec:SatoTate}
In this section, we consider the implications of Theorem \ref{thm:EllipticOdd} to the study of distributions of the trace of Frobenius.

To prove that the weighted averages of $x_E$ for $E/\F_p$ are equidistributed with respect to the Sato--Tate measure as $p\to\infty$, Birch \cite{Birch} computed the moments as 
\[
\sum_{\substack{E/\F_{p}}} \frac{\tr(E)^{2k}}{\omega_E}\sim C_k p^k.
\]
Up to a normalization factor, these match the moments in Lemma \ref{lem:SatoTateMoments}. Following Birch's argument, we are now ready to prove the following weighted version of Theorem \ref{thm:SatoTateArithmetic}.

\begin{theorem}\label{thm:dist1}
Let $m\in\Z$ and $M,r\in\N$. The set $\{x_E:E\in\mathcal{E}_{m,M}(p^r)\}$ becomes equidistributed with respect to the Sato--Tate measure as $p\to\infty$. Specifically, we have 
\[
\lim_{p\to\infty} \operatorname{Pr}_{\Aut}\left(a\leq x_E\leq b: E\in \mathcal{E}_{m,M}(p^r)\right) = \mu_{\ST}([a,b]).
\]
\end{theorem}

\begin{proof}
 By \eqref{eqn:Sratio}, for $\nu\in\N_0$, the $\nu$-th weighted moment with respect to $x_E$ for $E\in \mathcal{E}_{m,M}(p^r)$ is 
\begin{equation}\label{eqn:munudef}
\mu_{\nu}(p^r):=\frac{1}{S_{m,M}(p^r)}\sum_{E\in \mathcal{E}_{m,M}(p^r)} \frac{x_E^{\nu}}{\omega_E} = \frac{S_{\nu,m,M}(p^r)}{2^{\nu} p^{\frac{r\nu}{2}} S_{m,M}(p^r)}.
\end{equation}
If $\nu$ is odd, then Theorem \ref{thm:EllipticOdd} implies that (see \eqref{eqn:SratioOdd} in particular)
\begin{equation}\label{eqn:muprnuodd}
\lim_{p\to\infty} \mu_{\nu}(p^r)=0.
\end{equation}
If $\nu=2k$ is even, then we plug in \eqref{eqn:evenEllipticpr} to obtain 
\begin{equation}\label{eqn:muprnueven}
\lim_{p\to\infty} \mu_{\nu}(p^r)=\frac{C_{k}}{2^{2k}}.
\end{equation}
Comparing \eqref{eqn:muprnuodd} and \eqref{eqn:muprnueven} with Lemma \ref{lem:SatoTateMoments}, we conclude that
\begin{equation}\label{eqn:eigenlimit}
\lim_{p\to\infty} \mu_{\nu}(p^r)= \delta_{2\mid \nu} \frac{C_{\frac{\nu}{2}}}{2^{\nu}}=\mu_{\operatorname{ST},\nu}.
\end{equation}
Since the moments $\mu_{\nu}(p^r)$ converge to $\mu_{\ST,\nu}$  and the distribution is uniquely determined by its moments \cite[Theorem 30.1 and Theorem 30.2]{Bil} (as pointed out by Birch \cite{Birch},  the Sato--Tate distribution satisfies the necessary conditions), we conclude that 
\[
\lim_{p\to\infty} \operatorname{Pr}_{\Aut}\left(a\leq x_E\leq b: E\in \mathcal{E}_{m,M}(p^r)\right) = \mu_{\ST}([a,b]).\qedhere
\]
\end{proof}

We now compare the weighted and unweighted distributions in order to obtain Theorem \ref{thm:SatoTateArithmetic}.

\begin{proof}[Proof of Theorem \ref{thm:SatoTateArithmetic}]
 Since we take $p\to\infty$, we may assume that $p\geq 5$. In order to obtain the claim, we need to show that
\begin{equation}\label{eqn:unweightedEmMpr}
\lim_{p\to\infty} \frac{1}{\# \mathcal{E}_{m,M}(p^r)} \sum_{E\in \mathcal{E}_{m,M}(p^r)} x_E^{\nu} =\mu_{\operatorname{ST},\nu}.
\end{equation}

By \eqref{eqn:eigenlimit} (plugging in \eqref{eqn:munudef} to the left-hand side), we have
\begin{equation}\label{eqn:dist1conclusion}
\lim_{p\to\infty} \frac{S_{\nu,m,M}(p^r)}{2^{\nu} p^{\frac{r\nu}{2}}S_{m,M}(p^r)}  =\mu_{\operatorname{ST},\nu}.
\end{equation}
Using Lemma \ref{lem:weightedunweighted} and then \eqref{eqn:SmMlower}, we have 
\[
 \frac{S_{\nu,m,M}(p^r)}{2^{\nu} p^{\frac{r\nu}{2}}S_{m,M}(p^r)}
 = \frac{\frac{1}{2}\sum_{E\in\mathcal{E}_{m,M}(p^r)} x_E^{\nu}}{S_{m,M}(p^r)}+O_{\nu,M,\varepsilon}\left(p^{(\varepsilon-1)r}\right).
\]
In the main term, we use Lemma \ref{lem:weightedunweighted} and note that $\sum_{E\in\mathcal{E}_{m,M}(p^r)} 1 =\#\mathcal{E}_{m,M}(p^r)$ to rewrite 
\begin{equation}\label{eqn:SmMprEmMpr}
S_{m,M}(p^r) = \frac{1}{2}\#\mathcal{E}_{m,M}(p^r)+O(1).
\end{equation}
Using \eqref{eqn:SmMlower}, we see that have $\#\mathcal{E}_{m,M}(p^r)\gg_{M,\varepsilon} p^{(1-\varepsilon)r}$ and hence 
\begin{align}
\nonumber \frac{S_{\nu,m,M}(p^r)}{2^{\nu} p^{\frac{r\nu}{2}}S_{m,M}(p^r)}&=\frac{\frac{1}{2}\sum_{E\in\mathcal{E}_{m,M}(p^r)} x_E^{\nu}}{\frac{1}{2}\#\mathcal{E}_{m,M}(p^r) +O(1)}+O_{\nu,M,\varepsilon}\left(p^{(\varepsilon-1)r}\right)\\
\label{eqn:weightedunweightedrewrite} &=\frac{\frac{1}{2}\sum_{E\in\mathcal{E}_{m,M}(p^r)} x_E^{\nu}}{\frac{1}{2}\#\mathcal{E}_{m,M}(p^r)\left(1+ O_{M,\varepsilon}\left(p^{(\varepsilon-1)r}\right)\right)}+O_{\nu,M,\varepsilon}\left(p^{(\varepsilon-1)r}\right).
\end{align}
Expanding the geometric series, we have 
\begin{align*}
\frac{\frac{1}{2}\sum_{E\in\mathcal{E}_{m,M}(p^r)} x_E^{\nu}}{\frac{1}{2}\#\mathcal{E}_{m,M}(p^r)\left(1+ O_{M,\varepsilon}\left(p^{(\varepsilon-1)r}\right)\right)}&= \frac{\frac{1}{2}\sum_{E\in\mathcal{E}_{m,M}(p^r)} x_E^{\nu}}{\frac{1}{2}\#\mathcal{E}_{m,M}(p^r)}\left(1+O_{M,\varepsilon} \left(p^{(\varepsilon-1)r}\right)\right)\\
&= \frac{\sum_{E\in\mathcal{E}_{m,M}(p^r)} x_E^{\nu}}{\#\mathcal{E}_{m,M}(p^r)}+O_{M,\varepsilon} \left(p^{(\varepsilon-1)r}\right),
\end{align*}
where in the last step we use the fact that $|\sum_{E\in\mathcal{E}_{m,M}(p^r)} x_E^{\nu}|\leq \#\mathcal{E}_{m,M}(p^r)$. Plugging back into \eqref{eqn:weightedunweightedrewrite} and rearranging, we obtain that 
\begin{equation*}
 \frac{\sum_{E\in\mathcal{E}_{m,M}(p^r)} x_E^{\nu}}{\#\mathcal{E}_{m,M}(p^r)}= \frac{S_{\nu,m,M}(p^r)}{2^{\nu} p^{\frac{r\nu}{2}}S_{m,M}(p^r)}+O_{\nu,M,\varepsilon} \left(p^{(\varepsilon-1)r}\right).
\end{equation*}
Thus by \eqref{eqn:dist1conclusion} we have
\[
\lim_{p\to\infty}  \frac{\sum_{E\in\mathcal{E}_{m,M}(p^r)} x_E^{\nu}}{\#\mathcal{E}_{m,M}(p^r)}=\mu_{\operatorname{ST},\nu}.\qedhere
\]
\end{proof}

By restricting the sample space to $\mathcal{E}_{m,M}$, Theorem \ref{thm:dist1} investigates the (weighted) conditional probability that the normalized traces lie in a given interval under the assumption that $E\in \mathcal{E}_{m,M}(p^r)$. It is also interesting to consider the distribution of normalized traces $x_E$ from $E\in\mathcal{E}_{m,M}$ within the larger sample space $\mathcal{E}(p^r)$. Equivalently, weighting the probabilities via the reciprocal of the size of the automorphism group as before, we consider the moments 
\begin{align}\nonumber
\frac{\sum_{\substack{E/\F_{p}\\ \tr(E)\equiv m\pmod{M}}} \frac{x_E^{\nu}}{\omega_E}}{S_{1,1}(p)}&= \frac{\frac{1}{2^{\nu}}\sum_{\substack{E/\F_{p}\\ \tr(E)\equiv m\pmod{M}}} \frac{\tr(E)^{\nu}}{\omega_E}}{p^{\frac{\nu}{2}}S_{1,1}(p)}\\
\label{eqn:overallratio}&= \frac{\frac{1}{2^{\nu}}\sum_{\substack{E/\F_{p}\\ \tr(E)\equiv m\pmod{M}}} \frac{\tr(E)^{\nu}}{\omega_E}}{p^{\frac{\nu}{2}}S_{m,M}(p)}  \frac{S_{m,M}(p)}{S_{1,1}(p)}
\end{align}
as $p\to\infty$. Since the limit of the first factor in \eqref{eqn:overallratio} exists by \eqref{eqn:eigenlimit}, the limit as $p\to\infty$ of the the above distribution exists if and only if 
\[
\lim_{p\to\infty}  \frac{S_{m,M}(p)}{S_{1,1}(p)}
\]
exists. The closely-related limits \eqref{eqn:PrEs} were computed in \cite[Theorem 1]{CastryckHubrechts}, where the authors gave a full description of the distribution, so one may obtain the asymptotic behavior by combining Theorem~\ref{thm:dist1} with \cite{CastryckHubrechts}. We restrict ourselves to certain cases where the above limit exists in order to get fixed multiples of the Sato--Tate distribution in the limit. In order to further relate the results in \cite[Theorem 1]{CastryckHubrechts} with the approach used here, we next describe how to determine the limits considered in \cite{CastryckHubrechts} via our method. For this, assume that $p$ lies in a given congruence class modulo $M$. For $j\in\N$ with $\gcd(j,2M)=1$, we therefore set
\[
c_{j,m,M}:=\lim_{\substack{p\to\infty\\ p\equiv j\pmod{4M^2}}} \frac{S_{m,M}(p)}{S_{1,1}(p)}.
\]
We reprove the conclusion of \cite[Theorem 1]{CastryckHubrechts}  that these limits all exist.
 
\begin{proposition}\label{prop:limexists}
For $m\in\Z$, $M\in\N$, and $j\in\N$ with $\gcd(j,2M)=1$, the limit $c_{j,m,M}$ exists.
\end{proposition}

\begin{remark}
The limits are not in general independent of $j$, so the overall limit over all $p$ does not (in general) exist. 
\end{remark}

Before proving Proposition \ref{prop:limexists}, we conclude a weighted version of Theorem \ref{thm:SatoTateArithmetic2}.

\begin{theorem}\label{thm:dist2}
Let $m\in\Z$ and $M\in\N$. Then for each $j\in\Z$ with $\gcd(j,2M)=1$ there exists a positive constant $c_{j,m,M}$ such that for $-1\leq a\leq b\leq 1$ we have 
\[
\lim_{\substack{p\to\infty\\ p\equiv j\pmod{4M^2}}} \operatorname{Pr}_{\Aut}\left(a\leq x_E\leq b\text{ and }E\in \mathcal{E}_{m,M}(p):E\in \mathcal{E}(p)\right) = c_{j,m,M}\mu_{\ST}([a,b]).
\]
\end{theorem}
\begin{proof}
Plugging Proposition \ref{prop:limexists} and \eqref{eqn:eigenlimit} into \eqref{eqn:overallratio}, we obtain 
\[
\lim_{\substack{p\to\infty\\ p\equiv j\pmod{4M^2}}} \frac{\sum_{\substack{E/\F_{p}\\ \tr(E)\equiv m\pmod{M}}} \frac{x_E^{\nu}}{\omega_E}}{S_{1,1}(p)} =\begin{cases}
c_{j,m,M} \frac{C_k}{2^{2k}} &\text{if }\nu=2k\text{ is even},\\
0&\text{if }\nu\text{ is odd}.
\end{cases}
\]
By Lemma \ref{lem:SatoTateMoments}, this matches the $\nu$-th moment of $c_{j,m,M}$ times the Sato--Tate measure. Since the distribution is uniquely determined by its moments, as in \cite[Theorem 30.1 and Theorem 30.2]{Bil}, 
\[
\lim_{\substack{p\to\infty\\ p\equiv j\pmod{4M^2}}} \operatorname{Pr}_{\Aut}\left(a\leq x_E\leq b\text{ and }E\in \mathcal{E}_{m,M}(p): E\in\mathcal{E}(p)\right)=c_{j,m,M}\mu_{\ST}([a,b]).
\]
Hence for $E\in\mathcal{E}_{m,M}(p)$ within the larger space $\mathcal{E}(p)$, we see that $x_E$ is equidistributed in $[-1,1]$ with respect to $c_{j,m,M}$ times the Sato--Tate measure as $p\to\infty$ with $p\equiv j\Pmod{4M^2}$. This gives the claim.
\end{proof}

We next conclude Theorem \ref{thm:SatoTateArithmetic2} from Theorem \ref{thm:dist2}.

\begin{proof}[Proof of Theorem \ref{thm:SatoTateArithmetic2}]
Since the $\nu$-th moment of $c \mu_{\operatorname{ST}}$ is $c \mu_{\operatorname{ST},\nu}$ for $c\in\R$, the claim is equivalent to showing that 
\begin{equation}\label{eqn:SatoTateArithmetic2Equiv}
\lim_{\substack{p\to\infty\\ p\equiv j\pmod{4M^2}}}  \frac{1}{\#\mathcal{E}(p)} \sum_{E\in \mathcal{E}_{m,M}(p)} x_E^{\nu} = c_{j,m,M} \mu_{\operatorname{ST},\nu}.
\end{equation}
By \eqref{eqn:unweightedEmMpr}, we have 
\[
\lim_{p\to\infty} \frac{1}{\#\mathcal{E}_{m,M}(p)}\sum_{E\in \mathcal{E}_{m,M}(p)} x_E^{\nu}=\mu_{\operatorname{ST},\nu}.
\]
Thus \eqref{eqn:SatoTateArithmetic2Equiv} is equivalent to 
\begin{equation}\label{eqn:SatoTateArithmetic2Equiv2}
\lim_{\substack{p\to\infty\\ p\equiv j\pmod{4M^2}}}\frac{\#\mathcal{E}_{m,M}(p)}{\#\mathcal{E}(p)}  = c_{j,m,M}.
\end{equation}
To show \eqref{eqn:SatoTateArithmetic2Equiv2}, we use \eqref{eqn:SmMprEmMpr} to rewrite 
\[
\frac{\#\mathcal{E}_{m,M}(p)}{\#\mathcal{E}(p)}  =\frac{S_{m,M}(p)+O(1)}{S_{1,1}(p)+O(1)}.
\]
We then plug in \eqref{eqn:SmMlower} to obtain 
\[
\frac{S_{m,M}(p)+O(1)}{S_{1,1}(p)+O(1)}= \frac{S_{m,M}(p)}{S_{1,1}(p)} \frac{1+O_{M,\varepsilon}\left(p^{\varepsilon-1}\right)}{1+O_{\varepsilon}\left(p^{\varepsilon-1}\right)}=\frac{S_{m,M}(p)}{S_{1,1}(p)}\left(1+O_{M,\varepsilon}\left(p^{\varepsilon-1}\right)\right).
\]
By Proposition \ref{prop:limexists}, if $p\equiv j\Pmod{4M^2}$, then the right-hand side becomes
\[
c_{j,m,M}\left(1+o_{M}(1)\right),
\]
yielding \eqref{eqn:SatoTateArithmetic2Equiv2}, and hence the claim.
\end{proof}
We finally prove Proposition \ref{prop:limexists}.
\begin{proof}[Proof of Proposition \ref{prop:limexists}]
We claim that there exist constants $a_{j,m,M}>0$ and $b_{j,m,M}\in\C$ only depending on $j,m,M$ such that for $p\equiv j\Pmod{4M^2}$ we have 
\begin{equation}\label{eqn:SmMp}
S_{m,M}(p)=a_{j,m,M} p + b_{j,m,M}+O_{m,M}\left(p^{\frac{1}{2}+\varepsilon}\right)=a_{j,m,M} p +O_{j,m,M}\left(p^{\frac{1}{2}+\varepsilon}\right).
\end{equation}
Assuming \eqref{eqn:SmMp}, for $p\equiv j\Pmod{4M^2}$ we have 
\[
\frac{S_{m,M}(p)}{S_{1,1}(p)} = \frac{a_{j,m,M} p +O_{j,m,M}\left(p^{\frac{1}{2}+\varepsilon}\right)}{a_{j,1,1} p+O_{j,m,M}\left(p^{\frac{1}{2}+\varepsilon}\right)}= \frac{a_{j,m,M} }{a_{j,1,1} } +O_{j,m,M}\left(p^{-\frac{1}{2}+\varepsilon}\right).
\]
Thus
\[
c_{j,m,M}=\lim_{\substack{ p\to\infty\\ p\equiv j\pmod{4M^2}}} \frac{S_{m,M}(p)}{S_{1,1}(p)} =\frac{a_{j,m,M} }{a_{j,1,1}}>0
\]
exists and is positive because $a_{j,m,M}$ and $a_{j,1,1}$ are positive. It therefore remains to prove \eqref{eqn:SmMp}. Plugging \cite[Lemma 4.2]{BKP} (note that $\varrho_{m,M}(p)=0$ because $p$ is not a square) into \eqref{eqn:SHErel} yields
\[
S_{m,M}(p) =\frac{1}{2}H_{m,M}(p) + O\left(p^{\frac{1}{2}+\varepsilon}\right).
\]
This reduces \eqref{eqn:SmMp} to showing the same claim for $H_{m,M}(p)$. By \cite[Lemma 3.2]{BKP}, we have $H_{m,M}(p)=G_{m,M}(p)$ and $G_{m,M}(p)$ is the $p$-th coefficient of $2\left[\mathcal{H},\theta_{m,M}\right]_0$ by \cite[Lemma 3.1]{BKP}. By \cite[Lemma 3.3]{BKP}, we have that 
\[
f:=\left(2\left[\mathcal{H},\theta_{m,M}\right] + \Lambda_{1,m,M}\right)\big|U_4
\]
is a weight two quasimodular form on $\Gamma_{4M^2,M}$. The $n$-th coefficient of $\Lambda_{1,m,M}|U_4$ is $O_{m,M}(n^{\frac{1}{2}+\varepsilon})$ because $t^2-s^2=4n$ implies that $t-s\leq 2\sqrt{n}$. 

Thus it remains to show that the $p$-th coefficient of $f$ has the desired shape. We write 
\[
f=E+g
\]
where $g$ is a cusp form and $E$ is in the space spanned by Eisenstein series (which are spanned by $E_2$ and holomorphic modular forms that are orthogonal to cusp forms)
 Using the bound of Deligne \cite{Deligne}, the $p$-th coefficient of $g$ is bounded by $O(p^{\frac{1}{2}+\varepsilon})$. To compute the $p$-th coefficient of $E$, we use a basis given by $E_2$ and a basis for the space of holomorphic Eisenstein series of weight two. We first note that, by \cite[(1.7)]{OnoBook},  
\[
M_2\left(\Gamma_{4M^2,M}\right) \subseteq M_2\left(\Gamma_1\left(4M^2\right)\right) \cong \bigoplus_{\chi}M_{2}\left(\Gamma_0\left(4M^2\right),\chi\right),
\]
where $\chi$ runs over characters modulo $4M^2$. A basis for the space of Eisenstein series in\\$M_2(\Gamma_0(4M^2),\chi)$ may be found in \cite[Theorem 4.6.2]{DiamondShurman}. These are enumerated by pairs of characters $\psi$ and $\varphi$ modulo $4M^2$ for which $\psi\varphi=\chi$ and $t\in\N$ such that $c(\psi)c(\varphi)t\mid4M^2$, where $c(\psi)$ is the conductor. Writing $E_{2,\psi,\varphi,t}$ for the corresponding Eisenstein series, there exists constants $\lambda_{\psi,\varphi,t}$ such that (note that in \cite[Theorem 4.6.2]{DiamondShurman}, the case $(\psi,\varphi,1)$ is omitted, but this precisely corresponds to a multiple of $E_2$, so this is included for our quasimodular Eisenstein series)
\begin{equation}\label{eqn:Eexpand}
E= \sum_{\substack{\psi,\varphi\pmod{4M^2}}} \sum_{t\mid \frac{4M^2}{c(\psi)c(\varphi)}} \lambda_{\psi,\varphi,t} E_{2,\psi,\varphi,t}.
\end{equation}
If either $\psi$ or $\varphi$ is non-trivial, then the $n$-th Fourier coefficient of $E_{2,\psi,\varphi,1}$ is given by $\sigma_{1,\psi,\varphi}(n)$, where
\[
\sigma_{k,\psi,\varphi}(n):=\sum_{d\mid n} \varphi(d)\psi\left(\frac{n}{d}\right) d^{k}
\]
and $E_{2,\psi,\varphi,t}:=E_{2,\psi,\varphi,1}\big|V_t$. Since $p$ is coprime to $4M^2$, for $\psi\varphi$ non-trivial only those with $t=1$ may contribute to the sum. Since $p$ is prime, the sum has at most two summands. Moreover, since $p\equiv j\Pmod{4M^2}$ and $\varphi$ and $\chi$ are characters modulo $4M^2$, we have
\[
\sigma_{1,\psi,\varphi}(p)= \varphi(p)p +\psi(p) = \varphi(j) p +\psi(j). 
\]
For $\psi$ and $\varphi$ both trivial and $t\neq 1$, we have that
\[
E_{2,1,1,t}:=E_2-t E_2\big|V_t
\]
while $E_{2,1,1,1}:=E_2$. Again using the fact that $p$ and $4M^2$ are coprime, the $p$-th coefficient of $E_{2,1,1,t}$ is $-24(p+1)$. 
Plugging into \eqref{eqn:Eexpand}, for $p\equiv j\Pmod{4M^2}$, the $p$-th Fourier coefficient of $E$ is 
\begin{multline*}
\sum_{\substack{\psi,\varphi\pmod{4M^2}\\ (\psi,\varphi)\neq (1,1)}} \sum_{t\mid \frac{4M^2}{c(\psi)c(\varphi)}} \lambda_{\psi,\varphi,t} \left(\varphi(j) p +\psi(j)\right) -24 \sum_{t\mid 4M^2} \lambda_{1,1,t}(p+1)\\
=\left(\sum_{\substack{\psi,\varphi\pmod{4M^2}\\ (\psi,\varphi)\neq (1,1)}} \sum_{t\mid \frac{4M^2}{c(\psi)c(\varphi)}} \lambda_{\psi,\varphi,t} \varphi(j) -24\sum_{t\mid 4M^2} \lambda_{1,1,t}\right)p\\
 + \sum_{\substack{\psi,\varphi\pmod{4M^2}\\ (\psi,\varphi)\neq (1,1)}} \sum_{t\mid \frac{4M^2}{c(\psi)c(\varphi)}} \lambda_{\psi,\varphi,t} \psi(j) -  24\sum_{t\mid 4M^2} \lambda_{1,1,t}.
\end{multline*}
The two sums only depend on $m$, $M$, and $j$, yielding the claim.
\end{proof}

\end{document}